\documentclass{amsart}
\usepackage{amsmath,amsfonts,amssymb,latexsym}
\usepackage{multirow}

\newcommand{\ord}{\mathrm{ord}}

\theoremstyle{plain}
\numberwithin{equation}{section}
\newtheorem{thm}{Theorem}[section]
\newtheorem{theorem}[thm]{Theorem}

\newtheorem{lemma}[thm]{Lemma}

\newtheorem{corollary}[thm]{Corollary}
\newtheorem{conjecture}[thm]{Conjecture}

\theoremstyle{remark}
\newtheorem{remark}{Remark}

\begin{document}

\markboth{L\'aszl\'o Szalay and Volker Ziegler}
{S-Diophantine quadruples with $\mathbf{S=\{2,q\}}$}

\title{S-Diophantine quadruples with $\mathbf{S=\{2,q\}}$}

\author{L\'aszl\'o Szalay}
\address{Institute of Mathematics\\
           Faculty of Forestry\\
           University of West Hungary\\
           H-9400 Sopron, Ady E. \'ut 5.\\
           Hungary}
           \email{laszalay@emk.nyme.hu}

%\thanks{The research is supported...}

\author{Volker Ziegler}
\address{Johann Radon Institute for Computational and Applied Mathematics (RICAM)\\
Austrian Academy of Sciences\\
Altenbergerstr. 69\\
A-4040 Linz, Austria}
\email{volker.ziegler\char'100ricam.oeaw.ac.at}

%\begin{history}
%\received{(Day Month Year)}
%\accepted{(Day Month Year)}
%\end{history}

\begin{abstract}
Let $S$ denote a set of primes and let $a_1,\ldots,a_m$ be positive distinct
integers. We call the $m$-tuple $(a_1,\ldots,a_m)$ an $S$-Diophantine tuple if  $a_ia_j+1=s_{i,j}$ are $S$-units for all $i\not=j$. In this paper, we
show that no $S$-Diophantine quadruple (i.e.~$m=4$) exists if $S=\{2,q\}$ with $q\equiv 3\; (\bmod\, 4)$ or $q<10^9$. For two arbitrary primes $p,q<10^5$ we 
gain the same result.
\end{abstract}

\subjclass[2010]{11D61,11D45}
\keywords{Diophantine equations; $S$-unit equations; $S$-Diophantine quadruples.}

\maketitle

\section{Introduction}

Gy\H{o}ry, S\'ark\"ozy and Stewart~\cite{Gyory:1996} considered products of the form
\[\xi=\prod_{a\in A, b\in B}(ab+1),\]
where $A$ and $B$ are given sets of positive integers, and found lower bounds for the number of prime factors of $\xi$ in terms of $|A|$ and $|B|$. They also studied the 
specific case $A=B$ with 
$|A|=3$ and conjectured that the largest prime factor of
$(ab+1)(ac+1)(bc+1)$
tends to infinity as $\max\{a,b,c\}\rightarrow \infty$. This conjecture was finally proved by Corvaja and Zannier~\cite{Corvaja:2003}, and independently by 
Hern{\'a}ndez and Luca~\cite{Hernandez:2003}. Since both papers essentially depend on the Subspace Theorem, their results are ineffective, i.e.~the proofs 
provide no effective lower bound for the largest prime factor in term of $\max\{a,b,c\}$. The situation changes when 
$|A|=4$ since an effective result is known due to Stewart and Tijdeman~\cite{Stewart:1997}.

In this paper, we reverse the problem: if the number $k$ of prime divisors is given, find an upper bound for the cardinality of $A$ such that the 
number of prime divisors of 
\[\prod_{a,b\in A}(ab+1)\]
does not exceed $k$.

We introduce the following notations. Let $S$ denote a finite, fixed set of primes. We call an $m$-tuple $(a_1,\ldots,a_m)$
with positive and pairwise distinct integers $a_i$ ($1\leq i\leq m$) an $S$-Diophantine $m$-tuple, if the integers 
$$a_ia_j+1=s_{i,j}$$ 
are $S$-units for all 
$1\leq i <j\leq n$. Note that the classical Diophantine $m$-tuples contain squares instead of $S$-units, i.e.
\begin{equation}\label{Eq:Dioph} a_ia_j+1=\square_{i,j}, \end{equation}
which explains the denomination $S$-Diophantine tuple. Several variants of the classical Diophantine 
$m$-tuples have already been considered. For example, Bugeaud and Dujella~\cite{Bugeaud:2003} examined $m$-tuples, where $\square_{i,j}$ in~\eqref{Eq:Dioph} 
are replaced by $k$-th powers, Dujella and 
Fuchs~\cite{Dujella:2004c} investigated a polynomial version. Later Fuchs, Luca and Szalay~\cite{Fuchs:2008,Luca:2008} replaced $\square_{i,j}$ by terms of 
given 
binary recurrence sequences (cf.~\cite{Fuchs:2008}), and in particular the Fibonacci sequence (cf.~\cite{Luca:2008}). For a general survey 
we recommend Dujella's web page on Diophantine tuples~\cite{Dujella:HP}. So the $S$-Diophantine tuples can be considered as a 
further variant of the classical problem. 

In case of $|S|=2$ the authors recently showed that under some technical restrictions no 
$\{p,q\}$-Diophantine quadruple exists if $C(\xi)<p<q<p^\xi$ holds for any $\xi>1$ and for some explicitly computable constant $C(\xi)$ 
(see~\cite{Szalay:2013}). 
They also proved that no $\{p,q\}$-Diophantine quadruple exists if $p\equiv q\equiv 3\;(\bmod \,4)$ (see~\cite{Szalay:2013a}). By the results above, together 
with the experiences of a computer verification, the following 
conjecture is well-founded (see again~\cite{Szalay:2013,Szalay:2013a}).

\begin{conjecture}\label{Con:quadruples}
Let $p$ and $q$ be distinct primes. Then no $\{p,q\}$-Diophantine quadruple exists.
\end{conjecture}

In this paper, we provide more evidence to confirm this conjecture. In particular, one of our two main results is

\begin{theorem}\label{T1}
There is no $S$-Diophantine quadruple if $S=\{2,q\}$ and $q\equiv 3 \;(\bmod\, 4)$. 
\end{theorem}

The tools and methods resemble what we applied in~\cite{Szalay:2013a}, but there are some important differences, too.
Although the present Diophantine equations almost coincide with those appearing in~\cite{Szalay:2013a}, but the
symmetry we exploited in the $7^{th}$  and $8^{th}$ sub-cases of Table 1 in~\cite{Szalay:2013a} does not exist here. Therefore the proofs are more complicated 
in this paper, further we also must extend our proof at several other places.

Unfortunately, the aforementioned method does not work for $q\equiv 1\;(\bmod \,4)$. However, using continued fractions we establish a method, which allows us 
to 
find  all $\{2,q\}$-Diophantine quadruples if $q$ is given, independently of the remainder of $q$ divided by 4. In addition, this 
method works for all pairs of primes $(p,q)$. Thus we can prove the following theorem.

\begin{theorem}\label{Th:pq_small}
 Let $S=\{2,q\}$ with $q<10^9$ or $S=\{p,q\}$ with $p,q<10^5$. Then no $S$-Diophantine quadruple exists.
\end{theorem}

In the next section we provide auxiliary lemmas which will be used for the proofs of Theorems \ref{T1} and \ref{Th:pq_small}. Section \ref{Sec:CF} will 
introduce the continued 
fraction approach corresponding to the problem. These results, in particular Lemma \ref{lem:CFracMethod}, are the backbone of the proof of Theorem 
\ref{Th:pq_small}. In 
Section \ref{Sec:Alg} we describe an algorithm to find all $S$-Diophantine quadruples for a fixed set $S=\{p,q\}$. In particular, we will discuss the case 
$\{p,q\}=\{2,3\}$ in detail and prove that no such quadruples exist. The final section of the paper completes the proof of Theorem \ref{T1}.

\section{Auxiliary results and the easy cases}

Suppose that the primes $p$ and $q$ are given. To determine all $\{p,q\}$-Diophantine quadruples is obviously equivalent to find all solutions to the system
\begin{align}\label{sys1}
ab+1=&\;p^{\alpha_1}q^{\beta_1}=:s_1, & bc+1=&\;p^{\alpha_4}q^{\beta_4}=:s_4, \nonumber \\
ac+1=&\;p^{\alpha_2}q^{\beta_2}=:s_2, & bd+1=&\;p^{\alpha_5}q^{\beta_5}=:s_5, \\
ad+1=&\;p^{\alpha_3}q^{\beta_3}=:s_3, & cd+1=&\;p^{\alpha_6}q^{\beta_6}=:s_6. \nonumber 
\end{align}
Computing $abcd$ in different ways we obtain the three $S$-unit equations
\begin{align*}
s_3s_4-s_2s_5=&\;s_3+s_4-s_2-s_5, \\
s_2s_5-s_1s_6=&\;s_2+s_5-s_1-s_6, \\
s_1s_6-s_3s_4=&\;s_1+s_6-s_3-s_4. 
\end{align*}
Taking the $p$-adic and $q$-adic valuations, respectively we obtain the following Lemma (see also \cite[Proposition 1]{Szalay:2013}).

\begin{lemma}\label{lem:equal_evaluation}
The smallest two exponents of each quadruple $(\alpha_2,\alpha_3,\alpha_4,\alpha_5)$, \linebreak $(\alpha_1,\alpha_2,\alpha_5,\alpha_6)$ and 
$(\alpha_1,\alpha_3,\alpha_4,\alpha_6)$ coincide. A similar statement  holds also for the $\beta$'s.
\end{lemma}

In the sequel, this lemma will be frequently used without referring to its precise parameters. The next useful result excludes some divisibility relations for 
$S$-Diophantine triples.

\begin{lemma}\label{lem:div}
Assume that $(a,b,c)$ is an $S$-Diophantine triple with $a<b<c$. If $ac+1$ and $bc+1$, then $ac+1 \nmid bc+1$.
\end{lemma}

\begin{proof}
This is exactly Lemma 2 in~\cite{Szalay:2013}. 
\end{proof}

If we suppose $a<b<c<d$, then Lemma~\ref{lem:div} immediately shows that $\alpha_2\neq \alpha_4$, $\alpha_3\neq \alpha_5$,
$\alpha_3\neq \alpha_6$ and $\alpha_5\neq \alpha_6$ as well as $\beta_2 \neq \beta_4$, $\beta_3 \neq \beta_5$, $\beta_3 \neq \beta_6$ and $\beta_5 \neq 
\beta_6$.
The conditions $a<b<c<d$ will be mainly assumed in Sections \ref{Sec:CF} and \ref{Sec:Alg}.

Now we deduce a few further restrictions on the exponents appearing in the prime factorization
of the $S$-units $s_i$ when $p=2$ and $q\equiv 3\;( \bmod \,4)$.

\begin{lemma}\label{lem:exp_zero}
Let $S=\{2,q\}$ with $q\equiv 3\;( \bmod \,4)$ and let $(a,b,c)$ be an $S$-Diophan\-tine triple.
Further, according to (\ref{sys1})
\[ab+1=2^{\alpha_1}q^{\beta_1}, \qquad ac+1=2^{\alpha_2}q^{\beta_2}, \qquad bc+1=2^{\alpha_4}q^{\beta_4}.\]
Then at least one of $\beta_1,\beta_2,\beta_4$ is zero.
\end{lemma}

\begin{proof}
See Lemma 2.2 in~\cite{Szalay:2013a}.
\end{proof}

\begin{lemma}\label{lem:odd}
Assume that $S=\{2,q\}$. If $a,b,c$ and $d$ satisfy~\eqref{sys1}, then all of them are odd.
\end{lemma}

\begin{proof}
Suppose, for example, that $a$ is even. Then $\alpha_1=\alpha_2=\alpha_3=0$. In the virtue of Lemma~\ref{lem:div}, the integers $b$, $c$ and $d$ must be odd. 
Applying 
three times Lemma~\ref{lem:exp_zero} for the triples $(a,b,c)$, $(a,b,d)$ and $(a,c,d)$, respectively, we see that $\beta_4=\beta_5=\beta_6=0$. But this 
contradicts Lemma~\ref{lem:div} for the triple $(b,c,d)$. 
\end{proof}

Since all of $a,b,c$ and $d$ are odd, their residues modulo $4$ must be $1$ or $3$. Up to permutations we distinguish the following five cases. The 
remainders 
of $(a,b,c,d)$ modulo $4$ are
\[(1,1,1,1)\;{\rm or}\; (1,1,1,3)\;{\rm or}\; (1,1,3,3)\;{\rm or}\; (1,3,3,3)\;{\rm or}\; (3,3,3,3).\]
Among them only the middle one means a real possibility.

\begin{lemma}\label{onlyonecanremain}
If $(a,b,c,d)$ is a $\{2,q\}$-Diophantine quadruple, then up to permutations their remainders modulo $4$ are $(1,1,3,3)$.
\end{lemma}

\begin{proof}
Assume that three remainders of $a,b,c$ and $d$ modulo 4 coincide, say $x\equiv y\equiv z\;(\bmod\, 4)$. Clearly,
$xy+1\equiv xz+1\equiv yz+1\equiv 2 \;(\bmod\, 4)$. Then the three equations
\[xy+1=2q^{\beta_i}, \quad xz+1=2q^{\beta_j}, \quad yz+1=2q^{\beta_k}\]
contradict Lemma~\ref{lem:div}.
\end{proof}

Suppose now that $a\equiv b\equiv1$ and $c\equiv d\equiv3$ modulo 4. Then we arrive at the system
\begin{align*}
ab+1=&\;2q^{\beta_1},& bc+1=&\;2^{\alpha_4}q^{\beta_4},\\
ac+1=&\;2^{\alpha_2}q^{\beta_2},& bd+1=&\;2^{\alpha_5}q^{\beta_5}, \\
ad+1=&\;2^{\alpha_3}q^{\beta_3}, & cd+1=&\;2q^{\beta_6},
\end{align*}
where $\alpha_i\ge2$ and $\beta_1,\beta_6\ge1$ hold. Consider Lemma~\ref{lem:exp_zero} for all triples from the set $\{a,b,c,d\}$. Switching the integers $a$, 
$b$ respectively $c$, $d$ if necessary, we may assume that $\beta_2=\beta_5=0$. Therefore we have proved

\begin{lemma}\label{lem:System}
Let $S=\{2,q\}$ and $q\equiv 3\;( \bmod\, 4)$. If the system
\begin{align}\label{sys3b}
ab+1=&\;2q^{\beta_1},& bc+1=&\;2^{\alpha_4}q^{\beta_4}, \nonumber \\
ac+1=&\;2^{\alpha_2}, & bd+1=&\;2^{\alpha_5}, \\
ad+1=&\;2^{\alpha_3}q^{\beta_3},& cd+1=&\;2q^{\beta_6} \nonumber 
\end{align}
with $\alpha_i\ge2$ ($2\leq i\leq 5$) and $\beta_1,\beta_6\ge1$ has no solution, then there exists no $S$-Diophantine quadruple.
\end{lemma}

The next lemma is useful in excluding very small $S$-Diophantine quad\-rup\-les.

\begin{lemma}\label{lem:low_abc}
 Let $S=\{p,q\}$ be a set of two primes. Assume that $(a,b,c,d)$ is an $S$-Diophantine quadruple with $a<b<c<d$. Then $ab\geq 3$ and $c\geq 5$.
\end{lemma}

\begin{proof}
 First suppose that $a=1$ and $b=2$. Hence $ab+1=3$. Without loss of generality we may assume that $p=3$, $q\ne3$. Clearly, by (\ref{sys1}) we have
 \begin{align*}
  & &2c+1=&\;3^{\alpha_4}q^{\beta_4},\\
  c+1=&\;3^{\alpha_2}q^{\beta_2},&2d+1=&\;3^{\alpha_5}q^{\beta_5},\\
	d+1=&\;3^{\alpha_3}q^{\beta_3},&cd+1=&\;3^{\alpha_6}q^{\beta_6}
 \end{align*}
and, in particular 
\[2\cdot 3^{\alpha_2}q^{\beta_2}-3^{\alpha_4}q^{\beta_4}=1=2\cdot 3^{\alpha_3}q^{\beta_3}-3^{\alpha_5}q^{\beta_5}.\]
The first equation yields either $\alpha_2=\beta_4=0$ or $\alpha_4=\beta_2=0$, $q\neq 2$, while the second one implies $\alpha_3=\beta_5=0$ or 
$\alpha_5=\beta_3=0$, $q\neq 2$ (otherwise $3|1$ or $q|1$). Obviously, $\alpha_6\ne0$, otherwise we find a triple that contradicts Lemma~\ref{lem:div}. 
Similarly, $\beta_6\ne0$. According to which two out of $\alpha_2,\alpha_3,\alpha_4$ and $\alpha_5$ 
vanish we obtain equations of the form
\begin{equation}\label{eq:qeq} A\cdot3^{\alpha_6}q^{\beta_6}-A=(q^x-1)(q^y-1)=q^{x+y}-q^x-q^y+1,\end{equation}
where $x,y\ge1$, $x\ne y$, $A=1,2$ or $4$, and
\begin{equation}\label{eq:3eq} B\cdot3^{\alpha_6}q^{\beta_6}-B=(3^u-1)(3^v-1)=3^{u+v}-3^u-3^v+1,\end{equation}
where $u,v\ge1$, $u\ne v$, and $B$ satisfies $AB=4$.

Clearly, equation~\eqref{eq:3eq} is solvable only if  $3^{\min\{u,v,\alpha_6\}}|B+1$. Thus $B=2$ and then $A=2$. But we arrived at a contradiction, since 
equation~\eqref{eq:qeq} is solvable only if  
$q^{\min\{x,y,\beta_6\}}|A+1=3$.  

Let us turn to the second statement of the lemma. It is easy to check that the only $S$-Diophantine triple  with $|S|=2$ and $c<5$ is $(a,b,c)=(1,2,4)$. But 
now $ab+1=3$ which is excluded by the first statement.
\end{proof}

The next lemma is linked to the continued fraction approach, which will be discussed in Section~\ref{Sec:CF}.

\begin{lemma}\label{lem:lb1}
Let $p$ and $q$ be two odd primes and assume that $q^c\|p^{\ord_q(p)}-1$ and
$q^z|p^x-1$ hold for some integer $z\geq c$. Then $\ord_q(p) q^{z-c}|x$. Moreover, if $q^z|p^x+1$, then $\frac{\ord_q(p)}2q^{z-c}|x$.

In case of $q=2$ we assume that $2^c\|p^{\ord_4(p)}-1$ and $2^z|p^x-1$ for some integer $z\geq c$. Then $\ord_4(p) 2^{z-c}|x$. Moreover, if $2^z|p^x+1$, then 
$z=0,1$ if $p\equiv 1\;(\bmod\, 4)$ and $2^{z-c}|x$ otherwise.  
\end{lemma}

\begin{proof}
 See~\cite[Lemma 1]{Szalay:2013} or~\cite[Section 2.1.4]{Cohen:NTI}. Note that the additional case $q=2$ needs a special treatment, and the induction step in 
the proof of~\cite[Lemma 1]{Szalay:2013} should start with $c=2$, i.e. with $p^x\equiv 1 \;(\bmod\, 4)$.
\end{proof}

We will also frequently use the following famous result on Catalan's equation due to Mih{\v{a}}ilescu~\cite{Mihailescu:2004} during the proof of Theorem 
\ref{T1}.

\begin{lemma}\label{lem:Catalan}
The only solution to the Diophantine equation
\[X^a-Y^b=1\]
in positive integers $X,Y,a,b$ with $a,b\geq 2$ is $3^2-2^3=1$.
\end{lemma}

\begin{remark}
In fact, we do not need the full result on Catalan's equation. We use only that the Diophantine equation $2^x-q^y=\pm 1$ has no solution 
with $x,y>1$ and $q>3$, which was already proved by Hampel \cite{Hampel:1960}.
\end{remark}

\section{The continued fraction approach}\label{Sec:CF}

Throughout this section we assume that $a<b<c<d$. The essential tool proving Theorem~\ref{Th:pq_small} is 

\begin{lemma}\label{lem:CFracMethod}
 Let $B\ge\log d$ and assume that for some real number $\delta>0$ we have
 \[\left|P\log p-Q\log q\right|>\delta\]
 for all convergents $P/Q$ to ${\log q}/{\log p}$ with $Q<{2B}/{\log q}$ and $P<{2B}/{\log p}$. 
 Then
 \[\log d<2B_1+u_q\log q+u_p\log p+\log\left(\frac{2B_1^2}{\log p\log q}\right),\]
 where
 \[B_1=\max\left\{\log\left(\frac{2}{\delta}\right),\log\left(\frac{8B}{\log p \log q}\right)\right\},\]
 and $u_p$ and $u_q$ are defined by $p^{u_p}\|q^{\ord_p(q)}-1$ and $q^{u_q}\|p^{\ord_q(p)}-1$, respectively.
\end{lemma}

The aim of this section is to prove Lemma \ref{lem:CFracMethod}. First, we show that under the assumptions of Lemma~\ref{lem:CFracMethod},
$\alpha_1,\alpha_2$ and $\beta_1,\beta_2$ are relatively small.

\begin{lemma}\label{lem:alpha12_beta12_small}
Under the assumptions of Lemma~\ref{lem:CFracMethod}, $\alpha_1,\alpha_2<{B_1}/{\log p}$ and $\beta_1,\beta_2<{B_1}/{\log q}$ follows, where
\[B_1=\max\left\{\log\left(\frac{2}{\delta}\right),\log\left(\frac{8B}{\log p \log q}\right)\right\}.\]
Moreover, $\log(ab+1)<B_1$ and $\log(ac+1)<B_1$ also hold.
\end{lemma}

\begin{proof}
 Put $P_1=\alpha_1+\alpha_6-\alpha_2-\alpha_5$ and $Q_1=\beta_1+\beta_6-\beta_2-\beta_5$. Then 
 \[\tilde S:=p^{P_1}q^{Q_1}-1=\frac{s_1s_6}{s_2s_5}-1=\frac{(d-a)(c-b)}{(ac+1)(bd+1)}<\frac{(d-a)(c-b)}{abcd}<\frac 1{ab}.\]
 By Lemma \ref{lem:low_abc}, $0<\tilde S<1/3$. Therefore the signs of $P_1$ and $Q_1$ are opposite. Put $P:=|P_1|$, $Q:=|Q_1|$. Note that due to 
a Taylor expansion of order two, together with the Lagrange remainder we see
$$|\log(1+x)|\leq \frac{31 |x|}{24},$$
provided $|x|\leq 1/3$. Taking logarithms we obtain
 \[|P\log p-Q\log q|<\frac{31}{24(p^{\alpha_1}q^{\beta_1}-1)}<\frac{2}{p^{\alpha_1}q^{\beta_1}},\]
and then
\begin{equation}\label{Leg1}
\left| \frac{P}{Q}-\frac{\log q}{\log p}\right|<\frac{2}{Q(\log p) p^{\alpha_1}q^{\beta_1}}.
\end{equation}
Obviously, $\log q/\log p$ is irrational. 

If the right hand side of \eqref{Leg1} is smaller than $1/(2Q^2)$, then, according to Legendre's theorem, $P/Q$ is a convergent of the ratio of the logarithms 
of 
$q$ and $p$. Therefore $\delta<2/(p^{\alpha_1}q^{\beta_1})$, and we conclude $p^{\alpha_1}q^{\beta_1}<2/\delta$. Subsequently, $\alpha_1<B_1/\log p$ and 
$\beta_1<B_1/\log q$ with $B_1=\log(2/\delta)$. Also, $\log(ab+1)<\log(2/\delta)$. 

The case
\[\frac{2}{Q(\log p) p^{\alpha_1}q^{\beta_1}}>\frac{1}{2Q^2}\]
implies $p^{\alpha_1}q^{\beta_1}<4Q/\log p$. Thus, by the conditions of Lemma \ref{lem:CFracMethod},
$ab+1=p^{\alpha_1}q^{\beta_1}<{8B}/(\log p \log q)$,
and 
\[\log(ab+1)<\log\left(\frac{8B}{\log p \log q}\right).\]
Now, $\alpha_1<B_1/\log p$ and $\beta_1<B_1/\log q$.

Repeating the same machinery for $(s_2s_5)/(s_3s_4)-1$, it yields the app\-rop\-ri\-ate bounds for $\alpha_2$ and $\beta_2$.
\end{proof}

\begin{corollary}
Under the assumptions and notations of Lemma \ref{lem:alpha12_beta12_small} we have $\alpha_4< 2B_1/\log p$ and $\beta_4<2B_1/\log q$.
\end{corollary}

\begin{proof}
The corollary is an immediate consequence of Lemma \ref{lem:alpha12_beta12_small}, since 
$\alpha_4 \log p \leq \log(bc+1)$, $\beta_4 \log q \leq \log(bc+1)$ and 
$$\log(bc+1)< \log(a^2bc+ab+ac+1)= \log(ab+1)+\log(ac+1)< 2B_1.$$
\end{proof}

In the next step we prove that $\alpha_1=\alpha_2$ cannot be the minimum among $\alpha_1,\alpha_2,\alpha_5$ and $\alpha_6$. Similarly, $\beta_1=\beta_2$ 
is not the minimal pair in the set $\{\beta_1,\beta_2,\beta_5,\beta_6\}$.

\begin{lemma}
Assume that 
\[\log d>B_2:=2B_1+u_q\log q+u_p\log p+\log\left(\frac{2B_1^2}{\log p\log q}\right).\]
Then neither $\alpha_1=\alpha_2\leq \alpha_5,\alpha_6$ nor $\beta_1=\beta_2\leq \beta_5,\beta_6$ holds.
\end{lemma}

\begin{proof}
By switching the roles of $p$ and $q$ if necessary, we may assume that $\alpha_1=\alpha_2\leq \alpha_5,\alpha_6$. On the other hand $\beta_2$ cannot be the 
minimum of the 
$\beta$'s otherwise it yields $ac<ab$. By Lemma~\ref{lem:div} we exclude the minimality of $\beta_5=\beta_6$. Since a sole 
minimum does not exist we conclude that either $\beta_1=\beta_5$ or $\beta_1=\beta_6$.

First consider the case when $\alpha_1=\alpha_2$ and $\beta_1=\beta_5$ are minimal. By cancelling equal powers, the equation 
$ab\cdot cd=ac\cdot bd$ turns into
\[p^{\alpha_6}q^{\beta_6}-1-p^{\alpha_6-\alpha_1}q^{\beta_6-\beta_1}=p^{\alpha_5}q^{\beta_2}-q^{\beta_2-\beta_5}-p^{\alpha_5-\alpha_2},\]
which gives $p^z|1-q^{\beta_2-\beta_5}$, where $z=\min\{\alpha_6-\alpha_1,\alpha_5-\alpha_2\}$. Assume for the moment that 
$\alpha_5-\alpha_2=z$. Hence, by Lemma~\ref{lem:lb1}, $p^{\alpha_5-\alpha_2-u_p}\leq B_1/\log q$ follows. Then, together with the bound for $\alpha_1$ and 
$\alpha_2$ provided by Lemma~\ref{lem:alpha12_beta12_small}, it implies
\[\alpha_5< \frac{B_1}{\log p}+u_p+\frac{\log\left(\frac{B_1}{\log q}\right)}{\log p}.\]
Consequently, $\beta_5=\beta_1\leq \frac{B_1}{\log q}$ yields $\log d< B_2$. Therefore $z=\alpha_6-\alpha_1$. As before, we obtain the upper bound
\[\alpha_6< \frac{B_1}{\log p}+u_p+\frac{\log\left(\frac{B_1}{\log q}\right)}{\log p}.\]
In order to show that $\alpha_1=\alpha_2=\alpha_3$ and $\beta_1=\beta_4=\beta_5$ are minimal among all $\alpha$'s and $\beta$'s, respectively, we consider the 
equation 
$$ac\cdot bd=(p^{\alpha_2}q^{\beta_2}-1)(p^{\alpha_5}q^{\beta_5}-1)=(p^{\alpha_3}q^{\beta_3}-1)(p^{\alpha_4}q^{\beta_4}-1)=ad\cdot bd,$$
and determine the minimal pairs among $\alpha_2,\alpha_3,\alpha_4, \alpha_5$ and $\beta_2,\beta_3,\beta_4, \beta_5$, 
res\-pectively. By the assumption $\log d>2B_1$, together with $\beta_5\leq B_1/\log q$, we deduce that $\alpha_5>B_1/\log 
p>\alpha_2$. Hence, obviously, $\alpha_5$ is not minimal. So we obtain either $\alpha_2=\alpha_3< B_1/\log p$ or $\alpha_3=\alpha_4< \alpha_2< B_1/\log 
p$. In any case $\alpha_3<B_1/\log p$ follows, and since $\log d>2B_1$ we get $\beta_3>B_1/\log q>\beta_2$. Thus $\beta_3$ cannot be minimal. Therefore we 
conclude $\beta_1=\beta_5=\beta_4$. When $\alpha_3=\alpha_4\leq \alpha_2$ holds we arrive at a contradiction by $ac>bc$. Subsequently, 
$\alpha_1=\alpha_2=\alpha_3$ and $\beta_1=\beta_5=\beta_4$ are minimal among all $\alpha$'s and $\beta$'s, respectively. 

Next we consider the equation $ab\cdot cd=ad\cdot bc$, which leads to
\[p^{\alpha_6}q^{\beta_6}-1-p^{\alpha_6-\alpha_1}q^{\beta_6-\beta_1}=p^{\alpha_4}q^{\beta_3}-p^{\alpha_4-\alpha_1}-q^{\beta_3-\beta_1}.\]
Therefore $1-p^{\alpha_4-\alpha_1}|q^{\min\{\beta_6-\beta_1,\beta_3-\beta_1\}}$, and the application of Lemma~\ref{lem:lb1} yields
\[q^{\min\{\beta_6-\beta_1,\beta_3-\beta_1\}-u_q}\leq \alpha_4-\alpha_1<\frac{2B_1}{\log p}.\]
This gives an upper bound either for $\beta_3$ or for $\beta_6$, namely
\[\min\{\beta_3,\beta_6\}< \frac{B_1}{\log q}+u_q+\frac{\log\left(\frac{B_1}{\log p}\right)}{\log q}. \]
From the bounds for $\alpha_3$ and $\alpha_6$ we obtain (in any case) the contradiction
\[B_2<\log d< 2B_1+u_q\log q+u_p\log p+\log\left(\frac{2B_1^2}{\log p\log q}\right)=B_2.\]

The case $\alpha_1=\alpha_2$ and $\beta_1=\beta_6$ can be treated similarly. This time the equation $ab\cdot cd=ac\cdot bd$ admits
\[p^{\alpha_6}q^{\beta_1}-1-p^{\alpha_6-\alpha_1}=p^{\alpha_5}q^{\beta_2+\beta_5-\beta_1}-q^{\beta_2-\beta_1}-p^{\alpha_5-\alpha_2}q^{\beta_5-\beta_1},\]
which implies $p^z|1-q^{\beta_2-\beta_1}$, where $z=\min\{\alpha_6-\alpha_1,\alpha_5-\alpha_2\}$. The case $z=\alpha_6-\alpha_1$ cannot 
hold since $\alpha_6\le\alpha_5$ and $\beta_6\leq \beta_5$ contradict $bd<cd$. Therefore $z=\alpha_5-\alpha_2$, and by Lemma~\ref{lem:lb1} we obtain  the upper 
bound
\begin{equation}\label{ieq:alpha5} \alpha_5< \frac{B_1}{\log p}+u_p+\frac{\log\left(\frac{B_1}{\log q}\right)}{\log p}.\end{equation}
Consider the equation 
$$ac\cdot bd=(p^{\alpha_2}q^{\beta_2}-1)(p^{\alpha_5}q^{\beta_5}-1)=(p^{\alpha_3}q^{\beta_3}-1)(p^{\alpha_4}q^{\beta_4}-1)=ad\cdot bd$$
to see which pairs are minimal in the sets $\{\alpha_2,\alpha_3,\alpha_4, \alpha_5\}$ and 
$\{\beta_2,\beta_3,\beta_4, \beta_5\}$, respectively. Note that $\alpha_2$ is an element of the minimal pair in the set $\{\alpha_2,\alpha_3,\alpha_4, 
\alpha_5\}$. By 
inequality \eqref{ieq:alpha5} we deduce $\beta_5>B_1/\log q$, otherwise we obtain $\log d <B_2$. Similarly, 
$\beta_6=\beta_1<B_1/\log q$ yields $\alpha_6>B_1/\log p$.
%, otherwise we would obtain $\log d <B_2$. 
Thus neither $\beta_5$ 
nor $\alpha_6$, but either $\beta_2=\beta_3$ or $\beta_3=\beta_4<\beta_2$ are minimal. However, in any case we have $\beta_3\leq 
\beta_2<B_1/\log q$ and furthermore we have $\alpha_3>B_1/\log p$, otherwise we get $\log d<2B_1<B_2$. Subsequently, $\alpha_3$ cannot be minimal. Since $\alpha_2=\alpha_4$ yields a contradiction to Lemma \ref{lem:div} (see the discussion after the lemma) we may also exclude the minimality of  
$\alpha_4$. Finally, we conclude $\alpha_1=\alpha_2=\alpha_5<\alpha_4$. Now take the equation 
$$ab\cdot cd=(p^{\alpha_1}q^{\beta_1}-1)(p^{\alpha_6}q^{\beta_6}-1)=(p^{\alpha_3}q^{\beta_3}-1)(p^{\alpha_4}q^{\beta_4}-1) =ad\cdot bc$$
and determine the minimal pair among $\alpha_1,\alpha_3,\alpha_4,\alpha_6$. Obviously, $\alpha_1$ is included in the minimal pair. But, neither  
$\alpha_3$ nor $\alpha_4$ nor $\alpha_6$ is minimal since we obtain either $\alpha_1=\alpha_2=\alpha_5=\alpha_3$ or $\alpha_1=\alpha_2=\alpha_5=\alpha_6$ 
or $\alpha_1=\alpha_2=\alpha_4$. Observe, that all three cases are excluded by Lemma~\ref{lem:div}.
\end{proof}

\begin{lemma}
Let $\log d>B_2$. Then neither $\alpha_2\leq \alpha_1,\alpha_5,\alpha_6$ nor $\beta_2\leq \beta_1,\beta_5,\beta_6$ holds.
\end{lemma}

\begin{proof}
Because of the symmetry between $p$ and $q$ it is enough to show that $\alpha_2$ cannot be the minimum. Contrary, assume that 
$\alpha_2$ is the smallest one. Thus $\alpha_2=\alpha_5$ or $\alpha_2=\alpha_6$ since the case $\alpha_1=\alpha_2$ was excluded by the previous lemma. 
Clearly, $ac<ab$ guarantees that $\beta_2$ cannot be minimal among the $\beta$'s. 
Therefore we have the two possibilities 
$\alpha_2=\alpha_5$ and $\beta_1=\beta_6$ or $\alpha_2=\alpha_6$ and $\beta_1=\beta_5$. 
Indeed, the co-minimality of $\alpha_5$ and $\beta_5$ gives $bd<ab$, a contradiction. A similar argument is true for $\alpha_6$ and $\beta_6$ with $cd<ab$.
Moreover, $\beta_5=\beta_6$ contradicts Lemma~\ref{lem:div} applied to the triple $(b,c,d)$.

First, we treat the case $\alpha_2=\alpha_5$ and $\beta_1=\beta_6$. Therefore consider the equation 
$$ab\cdot cd=(p^{\alpha_1}q^{\beta_1}-1)(p^{\alpha_6}q^{\beta_6}-1)=(p^{\alpha_3}q^{\beta_3}-1)(p^{\alpha_4}q^{\beta_4}-1) =ad\cdot bc$$
to find the minimal pair among $\alpha_1,\alpha_3,\alpha_4,\alpha_6$ and $\beta_1,\beta_3,\beta_4,\beta_6$, respectively. Since $\beta_1=\beta_6<B_1/\log q$ and 
$\alpha_2=\alpha_5<B_1/\log p$ we have $\alpha_6>B_1/\log p$ and $\beta_5>B_1/\log q$, otherwise $\log d <2B_1<B_2$ hold. Therefore neither $\alpha_6$ nor $\beta_5$ is minimal. 
When $\alpha_3$ possesses the minimal property we have 
$\alpha_2=\beta_5<B_1/\log p$ and then $\beta_3>B_1/\log q$. In view of the equation 
$$ac\cdot bd=(p^{\alpha_2}q^{\beta_2}-1)(p^{\alpha_5}q^{\beta_5}-1)=(p^{\alpha_3}q^{\beta_3}-1)(p^{\alpha_4}q^{\beta_4}-1) =ad \cdot bc$$ 
neither $\beta_3$ nor $\beta_5$ can be minimal among $\beta_2,\beta_3,\beta_4$ and $\beta_5$. Thus $\beta_2=\beta_4$, but this is a 
contradiction to Lemma \ref{lem:div} (see the discussion below Lemma \ref{lem:div}). Consequently, $\alpha_3$ is not the minimum, therefore we get 
$\alpha_1=\alpha_4$. 
After simplifying the equation $ab\cdot cd=ad\cdot bc$ we obtain 
\[p^{\alpha_6}q^{\beta_6}-1-p^{\alpha_6-\alpha_1}=p^{\alpha_3}q^{\beta_3+\beta_4-\beta_1}-q^{\beta_4-\beta_1}-p^{\alpha_3-\alpha_1}q^{\beta_3-\beta_1}.\]
Therefore $p^z|q^{\beta_4-\beta_1}-1$, where $z=\min\{\alpha_6-\alpha_1,\alpha_3-\alpha_1\}$. Then, due to Lemma~\ref{lem:lb1}, 
\[p^{z-u_p}\leq \beta_4-\beta_1\leq 2B_1.\]
Assuming $z=\alpha_6-\alpha_1$, it leads to a contradiction. Indeed,
\[\alpha_6\leq \frac{B_1}{\log p}+u_p+\frac{\log\left(\frac{B_1}{\log q}\right)}{\log p}\]
holds, and together with  $\beta_1=\beta_6<B_1/\log q$ we obtain $\log d< \log (cd+1)<B_2$. Therefore we may assume that $z=\alpha_3-\alpha_1$ and we obtain
\[\alpha_3\leq \frac{B_1}{\log p}+u_p+\frac{\log\left(\frac{B_1}{\log q}\right)}{\log p}.\]
Since $\alpha_3$ is ``small'' we deduce that $\beta_3$ is ``large'', i.e. $\beta_3>B_1/\log q$, otherwise $\log d<B_2$ follows. Thus 
$\beta_3$ cannot be minimal. Consider the equation 
$$ac\cdot bd= (p^{\alpha_2}q^{\beta_2}-1)(p^{\alpha_5}q^{\beta_5}-1)=(p^{\alpha_3}q^{\beta_3}-1)(p^{\alpha_4}q^{\beta_4}-1)=ad\cdot bc$$
and determine which pair is minimal among $\beta_2,\beta_3,\beta_4,\beta_5$. 
By the discussion above, both $\beta_3$ and $\beta_5$ are not minimal and therefore $\beta_2=\beta_4$ is. But this is a contradiction to Lemma~\ref{lem:div} 
applied to the triple $(a,b,c)$.

Now let us turn to the case $\alpha_2=\alpha_6$ and $\beta_1=\beta_5$. We consider the equation 
$$ab\cdot cd=(p^{\alpha_1}q^{\beta_1}-1)(p^{\alpha_6}q^{\beta_6}-1)=(p^{\alpha_3}q^{\beta_3}-1)(p^{\alpha_4}q^{\beta_4}-1)=ad \cdot bc$$
in order to determine which pair among 
$\alpha_1,\alpha_3,\alpha_4,\alpha_6$ and $\beta_1,\beta_3,\beta_4,\beta_6 $ is minimal, respectively. Due to Lemma~\ref{lem:alpha12_beta12_small}, $\alpha_1$ 
cannot be minimal. Therefore we deduce that $\alpha_3=\alpha_4<\alpha_2=\alpha_6\leq\alpha_1$ since all other cases yield contradictions to Lemma~\ref{lem:div}. 
Concerning the $\beta$'s, if $\beta_1$ is not minimal, then $ab$ is larger than $ad,bc$ or $cd$. Therefore $\beta_1$ is minimal and we deduce 
$\beta_1=\beta_4=\beta_5$, since all other cases yield contradictions to Lemma~\ref{lem:div}. All together we have $\beta_4=\beta_1$ and $\alpha_4<\alpha_1$, 
which provide $bc<ab$, again a contradiction.
\end{proof}

Now we are ready to prove Lemma~\ref{lem:CFracMethod}

\begin{proof}[Proof of Lemma \ref{lem:CFracMethod}]
It is sufficient to show that the assumption $\log d>B_2$ leads to a contradiction. In view 
of foregoing lemmas we are left to examine two cases: 
\begin{itemize}
	\item $\alpha_1=\alpha_5$ and $\beta_1=\beta_6$ are minimal,
	\item $\alpha_1=\alpha_6$ and $\beta_1=\beta_5$ are minimal.
\end{itemize}
By switching the role of $p$ and $q$, if necessary, there is no restriction in assuming that $\alpha_1=\alpha_5$ and $\beta_1=\beta_6$ are minimal.
Now we intend to determine the minimal pair among $\alpha_2,\alpha_3,\alpha_4,\alpha_5$, therefore we consider the equation
$$ac\cdot bd=(p^{\alpha_2}q^{\beta_2}-1)(p^{\alpha_5}q^{\beta_5}-1)=(p^{\alpha_3}q^{\beta_3}-1)(p^{\alpha_4}q^{\beta_4}-1)=ad\cdot bc.$$
It is easy to see that only $\alpha_1=\alpha_4=\alpha_5$ or $\alpha_1=\alpha_5>\alpha_3=\alpha_4$ are possible. Indeed, Lemma~\ref{lem:div} exludes all the other cases.

First assume that $\alpha_1=\alpha_5>\alpha_3=\alpha_4$, which yields $\alpha_3,\alpha_5\leq B_1/\log p$. Thus 
$\beta_3,\beta_5\geq B_1/\log q$, otherwise $\log d <B_2$ follows. Now take the equation 
$$ac\cdot bd=(p^{\alpha_2}q^{\beta_2}-1)(p^{\alpha_5}q^{\beta_5}-1)=(p^{\alpha_3}q^{\beta_3}-1)(p^{\alpha_4}q^{\beta_4}-1)=ad\cdot bc$$
to see that neither $\beta_3$ nor $\beta_5$ is minimal among $\beta_2,\beta_3,\beta_4,\beta_5$. Therefore the minimal pair is $\beta_2=\beta_4$, but this 
contradicts Lemma~\ref{lem:div} by the triple $(a,b,c)$.

Turning to the case $\alpha_1=\alpha_4=\alpha_5<\alpha_2,\alpha_3$, note that $\beta_5>B_1/\log q$ holds because of  
$\alpha_5<B_1/\log q$ and $\log (bd+1)>\log d>B_2>2B_1$. In view of the equation 
$$ac\cdot bd=(p^{\alpha_2}q^{\beta_2}-1)(p^{\alpha_5}q^{\beta_5}-1)=(p^{\alpha_3}q^{\beta_3}-1)(p^{\alpha_4}q^{\beta_4}-1) =ad\cdot bc$$
we are looking for the minimal pair among $\beta_2,\beta_3,\beta_4,\beta_5$. Since $\beta_5$ is ``large'' and evidently not minimal, together with the conditions 
$\alpha_1=\alpha_4=\alpha_5<\alpha_2,\alpha_3$  the relation $\beta_2=\beta_3<\beta_4,\beta_5$ is the only possibility that does not contradict 
Lemma~\ref{lem:div} or the assumption 
$a<b<c<d$. Thus $\beta_3,\beta_6<B_1/\log q$.
Now we consider the equation $ab\cdot cd=ad\cdot bc$. After cancelling equal factors we obtain
\[p^{\alpha_6}q^{\beta_1}-1-p^{\alpha_6-\alpha_1}=p^{\alpha_3}q^{\beta_3+\beta_4-\beta_1}-p^{\alpha_3-\alpha_1}q^{\beta_3-\beta_1}-q^{\beta_4-\beta_1}.\]
Therefore $p^z|1-q^{\beta_4-\beta_1}$ follows, where $z=\min\{\alpha_6-\alpha_1,\alpha_3-\alpha_1\}$.
Here Lemma~\ref{lem:lb1} yields $p^{z-u_p}\leq \beta_4-\beta_1\leq 2B_1$, and we get
\[\min\{\alpha_6-\alpha_1,\alpha_3-\alpha_1\}\leq \frac{B_1}{\log p}+u_p+\frac{\log\left(\frac{B_1}{\log q}\right)}{\log p}.\]
But this bound, together with the bound for $\beta_3$ and $\beta_6$ yields the contradiction
\[B_2<\log d<\min\{\log(ad+1),\log(cd+1)\}< B_2.\]
\end{proof}

\section{The case of fixed $p$ and $q$}\label{Sec:Alg}

In view of the results and assumptions of the previous section we suppose that $a<b<c<d$ holds. Elaborating with the lower bounds for linear forms in 
logarithms of algebraic numbers due to Matveev~\cite{Matveev:2000} and Laurent et.~al.~\cite{Laurent:1995}, one can
obtain effectively computable upper bounds on $\log d$ in terms of $p$ and $q$ (see Stewart and Tijdeman~\cite{Stewart:1997}, or Szalay and Ziegler 
\cite{Szalay:2013}). If we look at the proof of~\cite[Lemma 7]{Szalay:2013} the authors concluded the three inequalities
\begin{eqnarray*}
\log\left(\frac cb \cdot \frac{bd+1}{cd+1}\right)
%=\log\left(\frac cb p^{\alpha_5-\alpha_6}q^{\beta_5-\beta_6}\right)
&\leq&\log\left(1+\frac 1{2d}\right)<\frac 1d, \\
\log\frac{(bd+1)(ac+1)}{(cd+1)ab}&<&\log\left(1+\frac2{ac}\right)<\frac 4c, \\
\log\frac{(ab+1)(cd+1)}{(ac+1)(bd+1)}&<&\log\left(1+\frac1{ab}\right)<\frac 2{ab}.
\end{eqnarray*}
Then they obtained 
\begin{eqnarray*}
1.690182\cdot 10^{10} \log c \log p \log q \left(2.1+\log\left(\frac{\log d}{\log c}\right)\right)&>&\log d,\\
1.690182\cdot 10^{10} \log (ab) \log p \log q \left(2.8+\log\left(\frac{\log d}{\log (ab)}\right)\right)&>& \log c-\log 4, \\
24.34 \log p \log q \left(2.08+\log\left(\frac{\log d}{\log p \log q}\right)\right)^2&>& \log(ab)-\log 2.
\end{eqnarray*}
According to Lemma~\ref{lem:low_abc}, we have 
\[\log(ab)-\log 2\geq \log(ab)\left(1-{\log 2}/{\log 3}\right)>0.369\log(ab)\]
and
\[\log(c)-\log 4\geq \log(c)\left(1-{\log 4}/{\log 5}\right)>0.138\log c.\]
Combining the two inequalities above it leads to
\begin{equation}\label{eq:lowbound_logd}
\begin{split}
 1.36\cdot10^{23} &(\log p \log q)^3(1.63+\log\log d)(2.71 +\log\log d)\\
 &\times (2.08-\log(\log p\log q)+\log\log d)^2>\log d.
\end{split} 
\end{equation} 

Now we restrict ourselves to determine all $\{p,q\}$-Diophantine quadruples in the case $(p,q)=(2,3)$. In particular, we prove the following

\begin{theorem}\label{Th:23tupel}
 There are no $\{2,3\}$-Diophantine quadruples.
\end{theorem}

\begin{proof}
Insert $p=2$ and $q=3$ in \eqref{eq:lowbound_logd} to get $\log d<1.6\cdot 10^{30}$. Then Lemma~\ref{lem:CFracMethod} 
provides $\log d<158.812$. Applying Lemma~\ref{lem:CFracMethod} again we obtain $\log d<21.966$, and then one more application gives $\log d<20.34$. Note that a 
further attempt to apply Lemma~\ref{lem:CFracMethod} does not yield a significant improvement (i.e.~the reduction of the upper bound is less than 0.1). 
Together with 
Lemma~\ref{lem:alpha12_beta12_small} we obtain $\alpha_1,\alpha_2\leq 9$ and $\beta_1,\beta_2\leq 6$. Moreover, the bound for $\log d$ leads to 
$\alpha_3,\alpha_4,\alpha_5,\alpha_6\leq 29$ and $\beta_3,\beta_4,\beta_5,\beta_6\leq 18$. All together it means $4900$ possibile pairs 
$(ab+1,ac+1)$. Since $a|\gcd(p^{\alpha_1}q^{\beta_1}-1,p^{\alpha_2}q^{\beta_2}-1)$ we can easily compute a list of possible $a$'s and corresponding $b$'s and 
$c$'s. By a computer search we obtain $2482$ triples with $a<b<c$. However, it is easy (for a computer) to check that only five of them are 
really $\{2,3\}$-Diophantine, 
namely the triples $(1, 5, 7), (1, 3, 5), (1, 7, 23), (1, 15, 17)$ and $(1, 31, 47)$. From $d=(p^{\alpha_6}q^{\beta_6}-1)/c$ with 
$0\leq \alpha_6\leq 29$ and $0\leq \beta_6\leq 18$ we obtain $344$ candidates to be $\{2,3\}$-Diophantine quadruples. But again a computer search 
shows that no candidate fulfills the conditions to be a $\{2,3\}$-Diophantine quadruple. 
\end{proof}

We implemented the ideas presented in the proof of Theorem~\ref{Th:23tupel} in Sage~\cite{sage} to find  all 
$S$-Diophantine quadruples if the set $S=\{p,q\}$ is given. 
Then we ran a computer search on all sets $S=\{2,q\}$ with $7\leq q <10^9$, $q$ prime and $q\equiv 1 \; (\bmod \,4)$, but no 
quadruples were found. This computer search, together with Theorem~\ref{T1} proves the first part of Theorem~\ref{Th:pq_small}. The computation was distributed 
on several kernels and computers. However, the total CPU time was about $121$ days and $16$ hours. Let us note that the proof of Theorem~\ref{T1} does not 
depend on this result.  

% 10600053 seconds in the case S={2,q}

Similarly, we ran a computer verification on all sets $S=\{p,q\}$ of odd primes $p$ and $q$ with $1<p,q<10^5$ such that $p\equiv q\equiv 3\; (\bmod \,4)$ does not hold. 
Even in this case we found no $S$-Diophantine quadruple, which proves the second part of Theorem~\ref{Th:pq_small}. Note that in a previous 
paper~\cite{Szalay:2013a} the authors proved that no $S$-Diophantine quadruples exist provided $p\equiv q\equiv 3\; (\bmod \,4)$. This computer search 
took in total approximately $27$ hours CPU time.

%98866 seconds in the general case

\section{Proof of Theorem~\ref{T1}}\label{Sec:Proofb}

In the proof of Theorem~\ref{T1} we are mainly concerned with System~\eqref{sys3b} (cf. Lemma~\ref{lem:System}). In order to get this specific symmetric 
form of the system we do not assume $a<b<c<d$ any longer. In view of Lemma~\ref{lem:System} we have a closer look at 
System~\eqref{sys3b} and assume that $\alpha_i\ge2$ and $\beta_1,\beta_6\ge1$. Moreover, suppose that $\alpha_5\ge\alpha_2$ and $\beta_6\ge\beta_1$. 
Therefore only the following relations are possible (see Table~\ref{Tab:Cases}). 

\begin{table}[ht]
\caption{List of cases}\label{Tab:Cases}
\begin{tabular}{|c|c|}
\hline $\alpha$ & $\beta$ \\\hline\hline
\multirow{3}*{$2\le\alpha_2=\alpha_3\leq \alpha_4$} & $1\le\beta_1=\beta_3\leq \beta_4$ \\\cline{2-2}
 & $1\le\beta_1=\beta_4\leq \beta_3$ \\\cline{2-2}
 & $0\le\beta_3=\beta_4<\beta_1$ \\\hline
 \multirow{3}*{$2\le\alpha_2=\alpha_4\leq \alpha_3$} & $1\le\beta_1=\beta_3\leq \beta_4$ \\\cline{2-2}
 & $1\le\beta_1=\beta_4\leq \beta_3$ \\\cline{2-2}
 & $0\le\beta_3=\beta_4<\beta_1$ \\\hline
 \multirow{3}*{$2\le\alpha_3=\alpha_4 < \alpha_2$} & $1\le\beta_1=\beta_3\leq \beta_4$ \\\cline{2-2}
 & $1\le\beta_1=\beta_4\leq \beta_3$ \\\cline{2-2}
 & $0\le\beta_3=\beta_4<\beta_1$ \\\hline
\end{tabular}
\end{table}

First, we show that only one of $\beta_1$ and $\alpha_2$ can be minimal.

\begin{lemma}
Assume that $\beta_1\leq \beta_i$ and $\alpha_2\leq \alpha_i$ are valid with $i=3,4,6$. Then system~\eqref{sys3b} possesses no solution.
\end{lemma}

\begin{proof}
Since $\beta_1\leq \beta_4$ and $\alpha_2\leq \alpha_4$, we see $a,b<c$. Moreover, we deduce that $ab+1|bc+1$ and $ac+1|bc+1$. But in view of 
Lemma~\ref{lem:div} both divisibility relations cannot hold.
\end{proof}

Therefore five sub-cases remain to examine.

\subsection{The case $\alpha_2=\alpha_3\leq \alpha_4$ and $\beta_3=\beta_4<\beta_1$}
Using the notation $\beta=\beta_3=\beta_4$ and $\alpha=\alpha_2=\alpha_3$, we obtain
\begin{align*}
ab+1=&\;2q^{\beta_1},& bc+1=&\;2^{\alpha_4}q^{\beta},\\
ac+1=&\;2^{\alpha},& bd+1=&\;2^{\alpha_5}, \\
ad+1=&\;2^{\alpha}q^{\beta},& cd+1=&\; 2q^{\beta_6}.
\end{align*}

The relation $ac+1<bc+1$ implies $a<b$. Similarly, $ac+1<ad+1$ entails $c<d$. Consequently, $\alpha_4<\alpha_5$. Moreover, considering the triple $(a,b,c)$ we 
have $c<b$, otherwise a contradiction to Lemma~\ref{lem:div} would appear. 

The equation
\[ad\cdot bc=\left(2^{\alpha}q^{\beta}-1\right)\left(2^{\alpha_4}q^{\beta}-1\right)=
\left(2^{\alpha}-1\right)\left(2^{\alpha_5}-1\right)=ac\cdot bd\]
modulo $2^{\alpha_4}$ yields
\begin{equation}\label{Eq:Mod} 
q^{\beta}\equiv 1 \;\;\;(\bmod \,2^{\alpha_4-\alpha}),
\end{equation}
and then $2^{\alpha_4-\alpha}|q^\beta-1$ follows. Thus $2^{\alpha_4-\alpha}\leq q^\beta-1$  holds.
At this point we separate the cases $\beta_1\geq 2\beta$ and $\beta_1< 2\beta$. 

First, assume that $\beta_1\geq 2\beta$. Taking
\[ad\cdot bc=\left(2^{\alpha}q^{\beta}-1\right)\left(2^{\alpha_4}q^{\beta}-1\right)=
\left(2q^{\beta_6}-1\right)\left(2q^{\beta_1}-1\right)=ab\cdot cd\]
modulo $q^{2\beta}$, we have
\[q^\beta\left(2^{\alpha_4}+2^\alpha\right)\equiv 0\;\;\;(\bmod\; q^{2\beta}).\]
Thus $q^\beta|2^{\alpha_4-\alpha}+1$ and $2^{\alpha_4-\alpha}\geq q^\beta-1$. Together with~\eqref{Eq:Mod}, we have 
$2^{\alpha_4-\alpha}=q^\beta-1$. By Lemma~\ref{lem:Catalan}, $\beta\leq 1$ holds since we excluded the case $q=3$ by Theorem~\ref{Th:23tupel}. The case 
$\beta=0$ is impossible, otherwise we have $2^{\alpha_4-\alpha} =0$. Therefore  $2^{\alpha_4-\alpha} =q-1$, and we deduce $\alpha_4-\alpha=1$ 
because $q\equiv 3 \; (\bmod 4)$. Hence we have $q=3$ which is excluded by Theorem~\ref{Th:23tupel}.

We turn now to the case $2\beta>\beta_1$. The equation
\[ad\cdot bc=\left(2^{\alpha}q^{\beta}-1\right)\left(2^{\alpha_4}q^{\beta}-1\right)=\left(2q^{\beta_1}-1\right)\left(2q^{\beta_6}-1\right)=ab\cdot cd\]
modulo $q^{\beta_1}$ leads to
\[q^\beta\left(2^{\alpha_4}+2^\alpha\right)\equiv 0\;\;\; (\bmod\, q^{\beta_1}).\]
Thus $q^{\beta_1-\beta}|2^{\alpha_4-\alpha}+1$ and $wq^{\beta_1-\beta}=2^{\alpha_4-\alpha}+1$ is valid with a suitable positive integer $w$.
A simple calculation shows that
\[
\frac 
ba=\frac{bc}{ac}=\frac{2^{\alpha_4}q^\beta-1}{2^\alpha-1}=2^{\alpha_4-\alpha}q^{\beta}+\frac{2^{\alpha_4-\alpha}q^\beta-1}{2^\alpha-1}>2^{\alpha_4-\alpha}q^{
\beta}.\]

If $w\ge3$, then $3q^{\beta_1-\beta}\leq 2^{\alpha_4-\alpha}+1$, and we obtain
\[\frac ba>2^{\alpha_4-\alpha}q^\beta\ge3q^{\beta_1}-q^\beta>3q^{\beta_1}-\frac 1q\, q^{\beta_1}\ge\frac {20}7 \,q^{\beta_1}. \]
But this is impossible since
\[2q^{\beta_1}=ab+1>b\geq \frac ba >\frac{20}7 \,q^{\beta_1}>2q^{\beta_1}.\] 

Suppose that $w=2$, that is $2q^{\beta_1-\beta}=2^{\alpha_4-\alpha}+1$. Clearly, $2^{\alpha_4-\alpha}+1$ is even if and only if $\alpha_4=\alpha$, hence  
$2q^{\beta_1-\beta}=2$ and $\beta_1=\beta$. Now $ab+1=2q^\beta$, $ac+1=2^\alpha$ and $bc+1=2^\alpha q^\beta$, which yield the equation
\begin{equation}\label{szer}
(ab+1)(ac+1)=2(bc+1).
\end{equation}
Only $a=1$ is possible, otherwise the left hand side of~\eqref{szer} is larger than the right hand side. But $a=1$ also leads to a contradiction since 
\[b\cdot c=2^{\alpha+1}q^\beta-2^\alpha-q^\beta+1=2^\alpha q^\beta-1=bc\]
gives $2^{\alpha}q^\beta=2^\alpha+q^\beta-2$, hence $2|q^\beta$.

Only the relation $q^{\beta_1-\beta}= 2^{\alpha_4-\alpha}+1$ remains to examine. Since we may exclude the case $q=3$ by Theorem~\ref{Th:23tupel}, 
Lemma~\ref{lem:Catalan} implies $\beta_1-\beta=1$. By the assumption $q\equiv 3\; (\bmod\, 4)$ we have $\alpha_4-\alpha=1$ and $q=3$, a contradiction. 

\subsection{The case $\alpha_2=\alpha_4\leq \alpha_3$ and $\beta_3=\beta_4<\beta_1$}
Denote by $\beta$ the value of $\beta_3=\beta_4$ and by $\alpha$ the integer $\alpha_2=\alpha_4$. 
We study the system
\begin{align*}
ab+1=&\;2q^{\beta_1},& bc+1=&\;2^{\alpha}q^{\beta},\\
ac+1=&\;2^{\alpha},& bd+1=&\;2^{\alpha_5}, \\
ad+1=&\;2^{\alpha_3}q^{\beta},& cd+1=&\;2q^{\beta_6}. 
\end{align*}
Comparing $ac+1$ and $bc+1$ we obtain $a<b$, hence $\alpha_3<\alpha_5$. Considering the equation
\[ad\cdot bc=\left(2^{\alpha_3}q^{\beta}-1\right)\left(2^{\alpha}q^{\beta}-1\right)=
\left(2^{\alpha}-1\right)\left(2^{\alpha_5}-1\right)=ac\cdot bd\]
modulo $2^{\alpha_3}$ it provides
\[2^{\alpha}q^{\beta}-1\equiv 2^{\alpha}-1 \;\;\;(\bmod 2^{\alpha_3}),\]
i.e. $q^{\beta}\equiv 1 \;\;\;(\bmod\, 2^{\alpha_3-\alpha})$.
Therefore  $2^{\alpha_3-\alpha}|q^{\beta}-1$ holds. We also have $c|c(b-a)=2^{\alpha}(q^{\beta}-1)$. Thus
$c|q^{\beta}-1$ and $c<q^\beta$. Hence $2^\alpha q^\beta=bc+1<bq^\beta+1$ yields $2^\alpha\leq b$.
Since $c$ is odd we obtain $c|\frac{q^{\beta}-1}{2^{\alpha_3-\alpha}}$. Therefore
$$
2^\alpha q^\beta=bc+1\le b\,\frac{q^\beta-1}{2^{\alpha_3-\alpha}}+1<b\,\frac{q^\beta}{2^{\alpha_3-\alpha}}+q^\beta
$$
implies $b>2^{\alpha_3}-2^{\alpha_3-\alpha}$.

On the other hand, $b|b(a-c)=q^\beta(2q^{\beta_1-\beta}-2^{\alpha})$. Hence
\begin{equation}\label{divisio}
b\mid q^{\beta_1-\beta}-2^{\alpha-1}.
\end{equation}
Observe that $q^{\beta_1-\beta}-2^{\alpha-1}$ is non-zero. The inequality $b\geq 2^{\alpha}$, together with~\eqref{divisio} provides 
$q^{\beta_1-\beta}>2^{\alpha-1}$, therefore $b\le q^{\beta_1-\beta}-2^{\alpha-1}$. Further, $c<a$ comes from the comparison of 
$ab+1$ and $bc+1$.

By $2q^{\beta_1}=ab+1\leq a(q^{\beta_1-\beta}-2^{\alpha-1})+1$ 
we deduce that
\[a\geq \frac{2q^{\beta_1}-1}{q^{\beta_1-\beta}-2^{\alpha-1}}.\]

Assuming $d>b$ we have
\begin{equation*}
2^{\alpha_3}q^\beta=\;ad+1>ab> (2^{\alpha_3}-2^{\alpha_3-\alpha})\frac{2q^{\beta_1}-1}{q^{\beta_1-\beta}-2^{\alpha-1}}=
2^{\alpha_3}q^\beta Q,
\end{equation*}
where
$$
Q=\left(1-\frac 
{1}{2^\alpha}\right)\frac{2q^{\beta_1}-1}{q^{\beta_1}-2^{\alpha-1}q^\beta}>\frac 34 \cdot\frac{2q^{\beta_1}-1}{q^{\beta_1}-1/2}=\frac 32>1
$$
means a contradiction. Recall that we assume $\alpha\ge2$. Thus $d<b$. But, together with $c<a$, this leads to $cd+1=2q^{\beta_6}<ab+1=2q^{\beta_1}$, 
which contradicts the assumption $\beta_1\leq \beta_6$.

\subsection{The case $\alpha_3=\alpha_4 < \alpha_2$ and $\beta_1=\beta_3\leq \beta_4$}
Let $\alpha=\alpha_3=\alpha_4$ and $\beta= \beta_1=\beta_3$.
Then we have to consider the system
\begin{align}\label{sys7}
ab+1=&\;2q^{\beta},& bc+1=&\;2^{\alpha}q^{\beta_4},\nonumber\\
ac+1=&\;2^{\alpha_2},& bd+1=&\;2^{\alpha_5}, \\
ad+1=&\;2^{\alpha}q^{\beta},& cd+1=&\;2q^{\beta_6} \nonumber
\end{align}
with $2\leq \alpha<\alpha_2\le\alpha_5$ and $1\le \beta\le\beta_4$. Due to \eqref{sys7}, $ab+1<ad+1$ holds. Thus $b<d$. This entails 
$\beta_4<\beta_6$ via $cb<cd$. Moreover, $ab+1<bc+1$ shows $a<c$.
The equation
$$
ad\cdot bc=(2^\alpha q^\beta-1)(2^\alpha q^{\beta_4}-1)=(2q^\beta-1)(2q^{\beta_6}-1)=ab\cdot cd
$$
modulo $q^{\beta_4}$ admits 
\[2^{\alpha-1}\equiv1 \;\;\;(\bmod \,q^{\beta_4-\beta}).\]
Thus $q^{\beta_4-\beta}\mid2^{\alpha-1}-1$ and $q^{\beta_4-\beta}\le2^{\alpha-1}-1$. 
We also consider the equation
\begin{equation}\label{csut1}
ad\cdot bc=(2^{\alpha} q^\beta-1)(2^\alpha q^{\beta_4}-1)=(2^{\alpha_2}-1)(2^{\alpha_5}-1)=ac\cdot bd,
\end{equation}
and distinguish two cases. First, suppose that $2\alpha\le\alpha_2$. Then~\eqref{csut1} modulo $2^{2\alpha}$ yields
$2^{\alpha}\mid q^{\beta_4-\beta}+1$. Thus $2^\alpha\le q^{\beta_4-\beta}+1\le 2^{\alpha-1}$ is a contradiction.

Now assume that $\alpha_2 < 2\alpha$. Then equation~\eqref{csut1} modulo $2^{\alpha_2}$ provides $2^{\alpha_2-\alpha}\mid q^{\beta_4-\beta}+1$, i.e. 
$w2^{\alpha_2-\alpha}=q^{\beta_4-\beta}+1$. We also note that
\begin{equation}\label{eq:ca}
\frac ca =\frac{bc}{ab}=\frac{2^\alpha 
q^{\beta_4}-1}{2q^\beta-1}=2^{\alpha-1}q^{\beta_4-\beta}+\frac{2^{\alpha-1}q^{\beta_4-\beta}-1}{2q^\beta-1}>2^{\alpha-1}q^{\beta_4-\beta}.
\end{equation}
If $w\geq 3$, then \eqref{eq:ca}, together with $c<2^{\alpha_2}$ (note that $ac+1=2^{\alpha_2}$) gives
\begin{eqnarray*}
\frac ca &>&2^{\alpha-1}q^{\beta_4-\beta}\ge2^{\alpha-1}(3\cdot2^{\alpha_2-\alpha}-1)=3\cdot2^{\alpha_2-1}-2^{\alpha-1} \\
&>&3\cdot2^{\alpha_2-1}-2^{\alpha_2-2}>\frac 54 c.
\end{eqnarray*}
Therefore we may assume that $w=1,2$.

First let $w=1$. This implies $2^{\alpha_2-\alpha}=q^{\beta_4-\beta}+1$. By Lemma \ref{lem:Catalan} we have 
$\beta_4-\beta\leq 1$. If $\beta_4=\beta$, then $\alpha_2-\alpha=1$ and by \eqref{sys7} we get 
\[(ab+1)(ac+1)=2q^\beta \cdot 2^{\alpha+1}= 4q^{\beta_4}2^{\alpha}=4(bc+1),\]
which provides $a=1$. Then 
$$
bc=(2q^\beta-1)(2^{\alpha+1}-1)=2^\alpha q^\beta-1,
$$
consequently $3\cdot 2^{\alpha}q^\beta+2=2^{\alpha+1}+2q^\beta$. Recalling $\alpha\geq 2$ and $\beta\geq 1$ we get
\[0=q^\beta2^{\alpha+1}+q^\beta(2^\alpha-2)-2^{\alpha+1}+2>0.\]

In case of $\beta_4-\beta=1$ we have $2^{\alpha_2-\alpha}=q+1$, moreover
$$
(ab+1)(ac+1)=\frac{2(q+1)}{q}(bc+1).
$$
By $2(q+1)/q\le 16/7$, this equality is possible only if $a=1$, i.e.~$q(b+1)(c+1)=2(q+1)(bc+1)$. After a few steps
$$
\left(b-\frac{q}{q+2}\right)\left(c-\frac{q}{q+2}\right)+\left(1-\frac{q^2}{(q+2)^2}\right)=0.
$$
follows. Obviously, on the left hand side all the three terms are positive, therefore we arrive at a contradiction.

The case $w=2$ remains providing $2^{\alpha_2-\alpha+1}=q^{\beta_4-\beta}+1$. Since $\alpha_2>\alpha$, by Lemma~\ref{lem:Catalan} we 
have $\beta_4-\beta=1$. Thus $2^{\alpha_2-\alpha+1}=q+1$, therefore
$$
(ab+1)(ac+1)=\frac{(q+1)}{q}(bc+1) 
$$
implies $a=1$ and $q(b+1)(c+1)=(q+1)(bc+1)$ or equivalently $bc+1=q(b+c)$. Replacing the integers $b$ and $c$ by the corresponding values on the right hand 
side of $bc+1=q(b+c)$ we find
$$
2^\alpha q^{\beta+1}=q\left(2q^\beta+2^{\alpha-1}(q+1)-2\right),
$$
and then
$$
2^{\alpha-2}=\frac{q^\beta-1}{2q^\beta-q-1}\in\mathbb{Z}.
$$
Obviously we have $\beta=1$ and $\alpha=2$. But, this yields $b=2q-1$ and $d=4q-1=2b+1$. Moreover, we have $bc=4q^2-1$ hence $c=2q+1=b+2$ and
\[2q^{\beta_6}=cd+1=(b+2)(2b+1)+1=(b+1)(2b+3).\]
Since $b+1$ and $2b+3$ are co-prime either $b+1|2$ or $2b+3|2$ holds. Thus $a=b=1$ or $b=-1$, both contradict our assumptions.

\subsection{The case $\alpha_3=\alpha_4 < \alpha_2$ and $\beta_1=\beta_4\leq \beta_3$}
Let $\alpha=\alpha_3=\alpha_4$ and $\beta= \beta_1=\beta_4$.
The corresponding situation can be described by
\begin{align*}
ab+1=&\;2q^{\beta},& bc+1=&\;2^{\alpha}q^{\beta},\\
ac+1=&\;2^{\alpha_2},& bd+1=&\;2^{\alpha_5}, \\
ad+1=&\;2^{\alpha}q^{\beta_3},& cd+1=&\;2q^{\beta_6}
\end{align*}
with the conditions $2\le\alpha\le\alpha_2\le\alpha_5$ and $1\le\beta\le\beta_3$.
We can easily deduce $b<d$ and $a<c$. Taking the equation 
$$ad\cdot bc=(2^{\alpha} q^\beta_3-1)(2^\alpha q^{\beta}-1)=(2^{\alpha_2}-1)(2^{\alpha_5}-1)=ab\cdot cd$$
modulo $q^{\beta_3}$, we see that
\[2^{\alpha-1}\equiv1\;\;\; (\bmod \,q^{\beta_3-\beta}).\]
That is $q^{\beta_3-\beta}\mid 2^{\alpha-1}-1$, further $q^{\beta_3-\beta}<2^{\alpha-1}$.

It is obvious that $b\mid b(c-a)=2^\alpha q^\beta-2q^\beta$, and then $b\mid2^{\alpha-1}-1$. Since $q$ and $b$ are co-prime we get 
$$
b\left| \frac{2^{\alpha-1}-1}{q^{\beta_3-\beta}}\right. .
$$
From 
$$2^\alpha q^\beta=bc+1\le \frac{2^{\alpha-1}-1}{q^{\beta_3-\beta}}c+1$$
we deduce
$c>2q^{\beta_3}-q^{\beta_3-\beta}$. 
Since 
$c\mid c(a-b)=2^\alpha(2^{\alpha_2-\alpha}-q^\beta)$,
we get
$c\mid (2^{\alpha_2-\alpha}-q^\beta)$. The assertion $c\le q^\beta$ leads to a contradiction via
\[bc+1\le (2^{\alpha-1}-1)q^\beta+1<2^\alpha q^\beta=bc+1.\]
Thus we have $c>q^\beta$. Therefore $2^{\alpha_2-\alpha}>q^\beta$ and $a>b$. Additionally, we have $2^{\alpha_2}=ac+1\le 
a(2^{\alpha_2-\alpha}-q^\beta)+1$, which implies
$$
a\ge\frac{2^{\alpha_2}-1}{2^{\alpha_2-\alpha}-q^\beta}.
$$

Now we show that neither $d>c$ nor $d<c$ is possible. Indeed, if $d>c$ then
$$
2^\alpha q^{\beta_3}=ad+1>ac\ge q^{\beta_3}\left(2-\frac{1}{q^\beta}\right)\frac{2^{\alpha_2}-1}{2^{\alpha_2-\alpha}-q^\beta}>2^\alpha q^{\beta_3},
$$ 
since $2-1/q^\beta>1$ and $({2^{\alpha_2}-1})/({2^{\alpha_2}-2^\alpha q^\beta})>1$. Assuming $d<c$, together with $b<a$ we get 
$2^{\alpha_5}=bd+1<ac+1=2^{\alpha_2}$, which contradicts our assumption $\alpha_2\le\alpha_5$.

\subsection{The case $\alpha_3=\alpha_4 < \alpha_2$ and $\beta_3=\beta_4<\beta_1$} Put $\alpha=\alpha_3=\alpha_4$ and $\beta=\beta_3=\beta_4$.
First suppose $c<a$. By $cd+1=2q^{\beta_6}\geq 2q^{\beta_1}=ab+1$ we deduce that $d>b$. Then
$ad+1=2^{\alpha}q^{\beta}=bc+1$ contradicts $c<a$ and $b<d$.
Now assume that $b<a$. We have $bd+1=2^{\alpha_5}\geq 2^{\alpha_2}=ac+1$, and then $d>c$. Thus we get
again a contradiction to $ad+1=2^{\alpha}q^{\beta}=bc+1$.

Therefore $a<b$ and $a<c$, further
\begin{align*}
 b|b(c-a)=&q^{\beta}(2^{\alpha}-2q^{\beta_1-\beta})>0, \quad \text{and}\\
 c|c(b-a)=&2^{\alpha}(q^{\beta}-2^{\alpha_2-\alpha})>0.
\end{align*}
Now $\gcd(b,q)=\gcd(c,2)=1$ implies $b<2^\alpha$ and $c< q^{\beta}$. Thus
$2^{\alpha}q^{\beta}>bc+1=2^{\alpha}q^{\beta}$, and the final contradiction is presented.

\section*{Acknowledgment}
The second author was supported by the Austrian Science Fund (FWF) under the project P 24801-N26.

\def\cprime{$'$}

%\bibliographystyle{abbrv}
%\bibliography{/home/vziegler/TUGraz/ziegler/NT}
\end{document}